\documentclass[12pt, reqno,draft]{amsart}
\usepackage{amsmath}
\usepackage{amssymb}
\usepackage{amsthm}
\usepackage{xcolor}
\usepackage{mathrsfs}
\usepackage[abbrev]{amsrefs}
\usepackage[utf8]{inputenc}
\usepackage{enumitem}
\usepackage{bm}
\usepackage{bbm}
\usepackage{graphicx}
\usepackage[all]{xy}

\newtheorem{thm}{}[section]
\newtheorem{theorem}[thm]{Theorem}

\newtheorem{lemma}[thm]{Lemma}
\newtheorem{proposition}[thm]{Proposition}

\theoremstyle{definition}

\theoremstyle{remark}
\newtheorem{remark}[thm]{Remark}

\newtheorem{question}[thm]{Question}

\numberwithin{equation}{section}
\allowdisplaybreaks
\makeatletter
\newcommand{\leqnos}{\tagsleft@true\let\veqno\@@leqno}
\newcommand{\reqnos}{\tagsleft@false\let\veqno\@@eqno}
\reqnos
\makeatother

\newcommand{\HD}{\ensuremath{H_d}}
\newcommand{\udf}{\ensuremath{\bm{\varphi_u}}}
\newcommand{\ldf}{\ensuremath{\bm{\varphi_l}}}
\newcommand{\ff}{\ensuremath{\bm{\varphi}}}
\newcommand{\primt}{\ensuremath{\bm{\tau}}}
\newcommand{\prim}{\ensuremath{\bm{\sigma}}}
\newcommand{\GG}{\ensuremath{\mathcal{G}}}
\newcommand{\leb}{\ensuremath{\bm{L}}}
\newcommand{\dom}{\ensuremath{\bm{\delta}}}
\newcommand{\ls}{\ensuremath{d^{\,0}}}
\newcommand{\XB}{\ensuremath{\mathcal{X}}}

\newcommand{\FF}{\ensuremath{\mathbb{F}}}
\newcommand{\XX}{\ensuremath{\mathbb{X}}}
\newcommand{\YY}{\ensuremath{\mathbb{Y}}}
\newcommand{\HH}{\ensuremath{\mathbb{H}}}
\newcommand{\Haar}{\ensuremath{\mathcal{H}}}
\newcommand{\ee}{\ensuremath{\bm{e}}}
\newcommand{\xx}{\ensuremath{\bm{x}}}
\newcommand{\uu}{\ensuremath{\bm{u}}}
\newcommand{\vv}{\ensuremath{\bm{v}}}
\newcommand{\ww}{\ensuremath{\bm{w}}}
\newcommand{\unc}{\ensuremath{\bm{k}}}

\newcommand{\UVS}{\ensuremath{\mathcal{E}}}
\newcommand{\NN}{\ensuremath{\mathbb{N}}}

\newcommand{\Fou}{\ensuremath{\mathcal{F}}}

\newcommand{\BV}{\ensuremath{\mathrm{BV}}}

\AtBeginDocument{\def\MR#1{}}

\begin{document}
%------------------------------------------------------------------------

\title[Almost greedy bases in superreflexive Banach spaces]{Lorentz spaces and embeddings induced by almost greedy bases in superreflexive Banach spaces}

\author[J. L. Ansorena]{Jos\'e L. Ansorena}
\address{Department of Mathematics and Computer Sciences\\
Universidad de La Rioja\\
Logro\~no\\
26004 Spain}
\email{joseluis.ansorena@unirioja.es}

\author[G. Bello]{Glenier Bello}
\address{Institute of Mathematics of the Polish Academy of Sciences\\
00-656 Warszawa\\
ul. \'{S}niadeckich 8\\
Poland}
\email{gbello@impan.pl}

\author[P. Wojtaszczyk]{Przemys{\l}aw Wojtaszczyk}
\address{Institute of Mathematics of the Polish Academy of Sciences\\
00-656 Warszawa\\
ul.\ \'Sniadeckich 8\\
Poland}
\email{wojtaszczyk@impan.pl}

\subjclass[2010]{46B15, 46B10, 46A35, 46A25, 41A65}

\keywords{superreflexive Banach space, almost greedy basis, Lorentz space}

%-----------------------
\begin{abstract}
The aim of this paper is to show that almost greedy bases induce tighter embeddings in superreflexive Banach spaces than in general Banach spaces. More specifically, we show that an almost greedy basis in a superreflexive Banach space $\XX$ induces embeddings that allow squeezing $\XX$ between two superreflexive Lorentz sequence spaces that are close to each other in the sense that they have the same fundamental function.
\end{abstract}
%-----------------------

\thanks{J.~L. Ansorena acknowledges the support of the Spanish Ministry for Science, Innovation, and Universities under Grant PGC2018-095366-B-I00 for \emph{An\'alisis Vectorial, Multilineal y Aproximaci\'on}. G. Bello and P. Wojtaszczyk
were supported by National Science Centre, Poland grant UMO-2016/21/B/ST1/00241.}

%-----------------------
\maketitle
%-----------------------

\section{Introduction}\noindent
Let $\XX$ be an infinite-dimensional separable Banach space (or, more generally, a quasi-Banach space) over the real or complex field $\FF$. Throughout this paper, by a \emph{basis} of $\XX$ we mean a sequence $\XB=(\xx_n)_{n=1}^\infty$ such that
\begin{itemize}[leftmargin=*]
\item it generates the entire space, i.e., the closure of its linear span is $\XX$,
\item there is a (unique) sequence $\XB^*=(\xx_{n}^*)_{n=1}^\infty$ in the dual space $\XX^{\ast}$ such that $(\xx_{n}, \xx_{n}^{\ast})_{n=1}^{\infty}$ is a biorthogonal system, and
\item it is semi-normalized, i.e., $\inf_n \Vert \xx_n\Vert>0$ and $\sup_n \Vert \xx_n\Vert<\infty$.
\end{itemize}
We will refer to the sequence $\XB^*$, which is a basis of the space it generates in $\XX^*$, as to the \emph{dual basis} of $\XB$. The numbers $\xx_n^*(f)$, $n\in\NN$, are the coefficients of $f\in\XX$ with respect to the basis $\XB$, and the map $\Fou\colon\XX\to \FF^\NN$ given by
\[
\Fou(f)=(\xx_n^*(f))_{n=1}^\infty,
\]
will be called the \emph{coefficient transform}.

A \emph{sequence space} will be a quasi-Banach space $\YY\subseteq\FF^\NN$ which contains the unit vectors, and for which the coordinate functionals are bounded. We say that $\XX$ embeds in the sequence space $\YY$ via $\XB$, and we write
\[
\XX \stackrel{\XB}\hookrightarrow \YY
\]
if $\Fou(f)\in\YY$ for all $f\in\XX$. In the reverse direction, we say that $\YY$ embeds in $\XX$ via $\XB$, and we write
\[
\YY \stackrel{\XB}\hookrightarrow \XX
\]
if the series $\sum_{n=1}^\infty a_n\, \xx_n$ converges for every $(a_n)_{n=1}^\infty\in\YY$. By the closed graph theorem, $\XX$ embeds in $\YY$ via $\XB$ if and only if the coefficient transform is a bounded operator from $\XX$ into $\YY$. Similarly, $\YY$ embeds in $\XX$ via $\XB$ if and only if the mapping
\[
f=(a_n)_{n=1}^\infty \mapsto \sum_{n=1}^\infty a_n \, \xx_n
\]
defines a bounded operator from $\YY$ into $\XX$. Serving in some situations as a tool to derive properties of bases, embeddings via bases are present in greedy approximation from the early years of the theory:
\begin{itemize}[leftmargin=*]
\item In \cites{Woj2003,DKK2003,AADK2019b} embeddings via bases are used to prove that certain bases are quasi-greedy.
\item The authors of \cite{AADK2016} used embedding via bases within their proof that certain Banach spaces have a unique greedy basis.
\item In \cites{BBG2017,BBGHO2018} embeddings via bases are used to achieve estimates for certain constants related to the greedy algorithm.
\end{itemize}

A sequence space $(\YY, \Vert \cdot\Vert_\YY)$ is said to be \emph{symmetric} if the mapping
\[
(a_n)_{n=1}^\infty\mapsto (\varepsilon_n a_{\pi(n)})_{n=1}^\infty
\]
defines an isometry on $\YY$ for every permutation $\pi$ of $\NN$ and every sequence $(\varepsilon_n)_{n=1}^\infty$ consisting of scalars of modulus one. Suppose we sandwich the quasi-Banach space $\XX$ between two symmetric sequence spaces $\YY_1$ and $\YY_2$ via the basis $\XB$ of $\XX$ as follows:
\[
\YY_1\stackrel{\XB}\hookrightarrow \XX \stackrel{\XB}\hookrightarrow \YY_2.
\]
Then, roughly speaking, the closeness between $\YY_1$ and $\YY_2$ is a qualitative estimate of the symmetry of $\XB$. A premier to determine whether $\YY_1$ and $\YY_2$ are close is to check whether they share the fundamental function. Given a sequence, mainly a basis, $\XB=(\xx_n)_{n=1}^\infty$ of a quasi-Banach space, its \emph{upper super-democracy function} (also known as the \emph{fundamental function}) and its \emph{lower super-democracy function} are respectively defined by
\begin{align*}
\udf(m)&
=\udf[\XB,\XX](m)=\sup\left\lbrace \left\Vert \sum_{n\in A} \varepsilon_n\, \xx_n \right\Vert \colon |A|\le m,\, |\varepsilon_n|=1 \right\rbrace,\\
\ldf(m)
&=\ldf[\XB,\XX](m)=\inf\left\lbrace \left\Vert \sum_{n\in A} \varepsilon_n\, \xx_n \right\Vert \colon |A|\ge m,\, |\varepsilon_n|=1 \right\rbrace.
\end{align*}
We define the fundamental function $\ff[\YY]$ of a symmetric sequence space $(\YY,\Vert\cdot\Vert_\YY)$ as that of its unit vector system $\UVS=(\ee_n)_{n=1}^\infty$. Notice that
\[
\ff[\YY](m):= \udf[\UVS,\YY](m)=\ldf[\UVS,\YY](m)=\left\Vert \sum_{n=1}^m \ee_n\right\Vert_\YY, \quad m\in\NN.
\]

Squeezing a quasi-Banach space between two symmetric sequence spaces with equivalent fundamental functions guarantees in a certain sense the optimality of the compression algorithms with respect to the basis (see \cite{Donoho1993}). This principle has led the specialists to address the task of obtaining embeddings via bases for Banach spaces and bases of interest in Approximation Theory. The authors of \cite{CDVPX1999} proved that if a function belongs to $\BV([0,1]^d)$, $d\ge 2$, then its coefficient transform with respect to the $L_1$-normalized Haar system $\Haar$ belongs to $\ell_{1,\infty}$. In other words, $\BV([0,1]^d)$ is squeezed between $\ell_1$ and $\ell_{1,\infty}$ via $\Haar$. In \cite{Woj2000} it is proved that any quasi-greedy basis $\XB$ of a Hilbert space $\HH$ is democratic, and $\HH$ can be squeezed between the Lorentz sequence spaces $\ell_{2,1}$ and $\ell_{2,\infty}$ via $\XB$. Later on, the authors of \cite{AlbiacAnsorena2016} used weighted Lorentz spaces to extend this result to general Banach spaces. Let us introduce this kind of sequence spaces. Given $0<q\le \infty$ and a weight $\ww=(w_n)_{n=1}^\infty$ whose primitive sequence $\prim=(s_m)_{m=1}^\infty$ is doubling (i.e., $s_m=\sum_{n=1}^m w_n$ and $\sup_m s_{m}/s_{\lceil m/2\rceil}<\infty$), we set
\[
\Vert f\Vert_{1,q,\ww}:=\left( \sum_{n=1}^\infty ( a_n s_n)^q \frac{w_n}{s_n}\right)^{1/q}, \quad f\in c_{0},\, f^*=(a_n)_{n=1}^\infty,
\]
with the usual modification if $q=\infty$. Recall that for a sequence $f=(\alpha_n)_{n=1}^\infty$ by $f^*$  we mean a nonincreasing permutation of the sequence $(|\alpha_n|)_{n=1}^\infty$. Then, we define the \emph{Lorentz sequence space} $d_{1,q}(\ww)$ as the quasi-Banach space consisting of all $f\in c_0$ such that $\Vert f\Vert_{1,q,\ww}<\infty$.

It is known \cite{DKKT2003}*{Theorem 3.3} that a basis is almost greedy if and only if it is quasi-greedy and democratic. In turn, if $\XB$ is a quasi-greedy democratic basis of a Banach space $\XX$, and the primitive sequence of $\ww$ is equivalent to $\udf[\XB,\XX]$, then
\begin{equation}\label{eq:SSAG}
d_{1,1}(\ww) \stackrel{\XB}\hookrightarrow \XX \stackrel{\XB}\hookrightarrow d_{1,\infty}(\ww)
\end{equation}
(see \cite{AlbiacAnsorena2016}*{Theorem 3.1}). Both results have been generalized to the more general setting of quasi-Banach spaces. Namely, the aforementioned characterization of almost greedy bases still hold (see \cite{AABW2021}*{Theorem 6.3}), and if $\XB$ is a quasi-greedy democratic basis of a $p$-Banach space $\XX$, $0<p\le 1$, then $\XX$ is squeezed between $d_{1,p}(\ww)$ and $d_{1,\infty}(\ww)$ via $\XB$ (see \cite{AABW2021}*{Theorem 4.13 and Corollary 9.13}).

Regardless the index $q\in(0,\infty]$, the fundamental function of $d_{1,q}(\ww)$ is equivalent to the primitive sequence of $\ww$. Thus, \eqref{eq:SSAG} exhibits a symmetry-like property of almost greedy bases of Lorentz sequence spaces. Besides, for a fixed weight $\ww$, the spaces $d_{1,q}(\ww)$, $0<q\le \infty$, set up a scale of quasi-Banach spaces in the sense that
\begin{equation*}
d_{1,q}(\ww) \subseteq d_{1,r}(\ww), \quad 0<q<r\le \infty.
\end{equation*}
Thus, squeezing a Banach space $\XX$ via a given almost greedy basis $\XB$ as
\begin{equation}\label{eq:SSAGImproved}
d_{1,q}(\ww) \stackrel{\XB}\hookrightarrow \XX \stackrel{\XB}\hookrightarrow d_{1,r}(\ww)
\end{equation}
with $1\le q \le r \le \infty$ and either $q>1$ or $r<\infty$ would exhibit that the basis gains in symmetry and, then, the thresholding greedy algorithm (TGA for short) with respect to it performs better. In this regard, we bring up the connection between embeddings involving Lorentz sequence spaces and Lebesgue-type inequalities introduced by Temlyakov \cites{Temlyakov1998,Temlyakov1998b}. The TGA $(\GG_{m})_{m=1}^{\infty}$ with respect to the basis $\XB$ of $\XX$ is the sequence of (nonlinear) operators from $\XX$ into $\XX$ defined for each $f\in\XX$ and $m\in\NN$ by choosing the first $m$ terms in decreasing order or magnitude from the formal series expansion $\sum_{n=1}^\infty\xx_n^*(f)\, \xx_n$, with the agreement that when two terms are of equal size we take them in the basis order. Then, to measure the performance of the greedy algorithm we define the \emph{$m$th Lebesgue constant} $\leb_m=\leb_m[\XB,\XX]$ for the TGA of the basis $\XB$ as the smallest constant $C$ such that
\[
\Vert f-\GG_{m}(f)\Vert\le C \Vert f-g\Vert
\]
for every linear combination $g$ of at most $m$ vectors of the basis. If $\XB$ is almost greedy, then its Lebesgue parameters grow as its \emph{conditionality constants} $(\unc_m)_{m=1}^\infty$ used to quantify how far is it from being unconditional (see \cite{GHO2013}*{Remark 4}). Specifically, the $m$th \emph{conditionality constant} of the basis $\XB$ is the number
\[
\unc_m=\unc_m[\XB,\XX] :=\sup_{|A|\le m} \Vert S_A\Vert, \quad m\in\NN,
\]
where $S_A\colon\XX\to \XX$ is the \emph{coordinate projection} with respect to the basis $\XB$ on the finite set $A$. If $\XX$ is squeezed between symmetric sequence spaces $\YY_1$ and $\YY_2$ via $\XB$, then
\[
\leb_m[\XB,\XX]\lesssim \dom_m[\YY_1,\YY_2], \quad m\in\NN,
\]
where
$\dom_m[\YY_1, \YY_2]$ is the smallest constant $C$ such that 
\[
\left\Vert\sum_{n=1}^m a_n\, \ee_n\right\Vert_{\YY_2} \le C \left\Vert\sum_{n=1}^m a_n\, \ee_n\right\Vert_{\YY_1}, \quad a_1\ge\cdots \ge a_n\ge\cdots \ge a_m\ge 0
\]
(see \cite{AAB2021}*{Theorem 1.5 and Lemma 5.1}). A routine check gives that, if $0<q\le r\le \infty$ and $\ww$ is a weight whose primitive sequence $\prim=(s_n)_{n=1}^\infty$ is doubling,
\[
\dom_m[d_{1,q}(\ww), d_{1,r}(\ww)]=(H_m[\prim])^{1/q-1/r}, \quad m\in\NN,
\]
where $H_m[\prim]=\sum_{n=1}^m w_n/s_n$. Since $H_m[\prim]\lesssim \log m$ for $m\ge 2$ (see \cite{AAB2021}*{\S5}), if the basis $\XB$ satisfies \eqref{eq:SSAGImproved}, then it would satisfy the Lebesgue-type inequality
\[
\leb_m[\XB,\XX]\lesssim (\log m)^{1/q-1/r}, \quad m\ge 2.
\]
Taking into account that almost greedy bases of Banach spaces fulfil \eqref{eq:SSAG}, we infer that they satisfy the Lebesgue-type inequality $\leb_m\lesssim \log m$ for $m\ge 2$. This estimate is well-known (see \cite{GHO2013}*{Remark 4}). More remarkable is that squeezing $\XX$ as in \eqref{eq:SSAGImproved} with indices $q$ and $r$ closer to each other than $1$ and $\infty$ would lead to improving it. Namely, we would obtain
\begin{equation}\label{eq:LPImproved}
\leb_m\lesssim (\log m)^a, \quad m\ge 2.
\end{equation}
for some $0<a<1$. We point out that, since there are almost greedy bases of Banach spaces with $\leb_m\approx \log m$ (see \cite{GHO2013}*{\S6}), it is hopeless trying to obtain \eqref{eq:SSAGImproved} with $1\le q \le r \le \infty$ and either $q>1$ or $r<\infty$ for almost greedy bases of general Banach spaces. Still, it is reasonable to wonder whether the geometry of the space could help to obtain better embeddings. To reinforce this guess, we mention that quasi-greedy bases of Hilbert (are almost greedy and) satisfy \eqref{eq:LPImproved} for some $a<1$ (see \cite{GW2014}*{Theorem 1.1}). Later on, this result was generalized to superreflexive spaces \cite{AAGHR2015}. So, it is natural to address the task of improving \eqref{eq:SSAG} for almost greedy bases of superreflexive Banach spaces. This paper is conducted toward reaching this goal:

\begin{theorem}\label{thm:main}
Let $\XB$ be an almost greedy basis of a superreflexive Banach space $\XX$, and let $\ww=(w_n)_{n=1}^\infty$ be the weight defined by
\[
\sum_{n=1}^m w_n=\udf[\XB,\XX](m), \quad m\in\NN.
\]
Then, there are $1<q<r<\infty$ such that
$
d_{1,q}(\ww) \stackrel{\XB}\hookrightarrow \XX \stackrel{\XB}\hookrightarrow d_{1,r}(\ww).
$
\end{theorem}
Theorem~\ref{thm:main} is optimal in the sense that, even if $\XX$ is a Hilbert space, there is neither a lower bound other that $1$ for the index $q$ nor a finite upper bound for the index $r$. This claim can be deduced from the fact that for each $0<a<1$ there is a quasi-greedy basis of a Hilbert space with $\unc_m\gtrsim (\log m)^a$ \cite{GW2014}*{Theorem 1.2}.

We conclude this introductory section by describing in some detail the contents of this paper. Section~\ref{sect:EQB} will be devoted to proving Theorem~\ref{thm:main}, which will be obtained as a consequence of an embedding theorem for quasi-greedy bases and certain duality techniques. Previously, in Section~\ref{sect:UC} we obtain new inequalities, which play a crucial role within our approach, for uniformly convex norms. Also with an eye in its application within Section~\ref{sect:EQB}, in Section~\ref{sect:SRLSS} we delve into the theory of Lorentz sequence spaces. In particular, we give a duality theorem that improves a theorem of Allen \cite{Allen1978}, and a characterization of superreflexive Lorentz sequence spaces that generalizes a theorem of Altshuler \cite{Altshuler1975}. As a matter of fact, Lorentz sequence spaces involved in Theorem~\ref{thm:main} are superreflexive. So, Theorem~\ref{thm:main} exhibits that any superreflexive Banach space $\XX$ equipped with an almost greedy basis can be sandwiched between symmetric spaces close to each other and of the same nature as $\XX$. In this regard, Theorem~\ref{thm:main} is the almost greedy counterpart of the following classical result on Schauder bases of superreflexive Banach spaces.

\begin{theorem}[\cite{GurGur1971}*{Theorems 1 and 2} and \cite{James1972}*{Theorems 2 and 3}]\label{thm:EJ}
Let $\XB$ is a semi-normalized Schauder basis of superreflexive Banach space $\XX$. Then there are $1<q<r<\infty$ such that $\XB$ $q$-Hilbertian and $r$-Besselian, that is,
\[
\ell_q \stackrel{\XB}\hookrightarrow \XX \stackrel{\XB}\hookrightarrow \ell_r.
\]
\end{theorem}

Finally, in Section~\ref{sect:OP} we present a selection of questions that arise from our work.

We use standard terminology and notation in Banach space theory as can be found, e.g., in \cites{AlbiacKalton2016}. For background on Lorentz sequence spaces and greedy-like bases in the general setting of quasi-Banach spaces, as well as for any terminology not explicitly settled here, we refer the reader to \cite{AABW2021}.

\section{Norm estimates in superreflexive spaces}\label{sect:UC}\noindent
A Banach space $\XX$ is \emph{superreflexive} if every Banach space finite representably in $\XX$ is reflexive. A classical result of Enflo \cite{Enflo1972} characterizes superreflexive spaces as those Banach spaces that can be endowed with an equivalent \emph{uniformly convex} norm, i.e., a norm such that for each $0< \varepsilon<2$ there is $\delta\in(0,1)$ with
\begin{equation}\label{UniformConvexCondition}
\Vert f\Vert, \; \Vert g \Vert \le 1 \text{ and } \Vert f- g\Vert \ge \varepsilon
\Longrightarrow
\Vert f+ g\Vert \le 2(1-\delta).
\end{equation}
For each $ \varepsilon\in[0,2]$, let denote by $\delta_\XX( \varepsilon)$ the largest $\delta$ such that \eqref{UniformConvexCondition} holds, that is,
\begin{equation}\label{eq:mc}
\delta_\XX( \varepsilon)=\inf\left\{1-\frac{\|f+g\|}{2} \colon \|f\|,\ \|g\| \le 1 \mbox{ and } \|f-g\| \ge \varepsilon \right\}.
\end{equation}
The mapping $ \varepsilon\mapsto\delta_\XX( \varepsilon)$, called \emph{modulus of convexity of the norm}, is a nondecreasing function with $\delta_\XX(0)=0$ and $\delta_\XX(2)=1$. The space $\XX$ is uniformly convex if and only if $\delta_\XX( \varepsilon)>0$ for $ \varepsilon >0$. It is known \cite{Figiel1976} that $\delta_\XX$ can be defined as well replacing inequalities with equalities in \eqref{eq:mc} (see also \cite{LinTza1979}).

Given $1<q<\infty$, any norm $\Vert \cdot\Vert$ satisfies the inequality
\[
\Vert f+g\Vert^q \le C ( \Vert f\Vert^q+\Vert g\Vert^q), \quad f,\, g \in\XX
\]
with $C=2^{q-1}$. If the norm is uniformly convex we can get $C<2^{q-1}$ under the assumption that $\Vert f-g\Vert\ge\varepsilon \max\{\Vert f\Vert, \Vert g\Vert\} $ for some $\varepsilon>0$ (see \cite{AAGHR2015}*{Proof of Lemma 2.1}). Our first result shows that we can even obtain $C=1$ under an extra assumption on the vectors $f$ and $g$, i.e., any uniformly convex norm satisfies, under certain restrictions, a triangle $q$-law for some $q>1$.

\begin{lemma}\label{Lemma1}
Let $\XX$ be a uniformly convex Banach space. Then, for every $0< \varepsilon<2$ there are $q>1$ and $0<\eta<1$ with the property that for every $f$, $g\in \XX$ with
\begin{enumerate}[label=(\roman*),leftmargin=*, widest=ii]
\item\label{cond:1} $\min\{\Vert f\Vert, \Vert g\Vert\}\ge (1-\eta)\max\{\Vert f\Vert, \Vert g\Vert\}$ and
\item\label{cond:2} $\|f-g\|\ge \varepsilon \max\{\Vert f\Vert, \Vert g\Vert\}$,
\end{enumerate}
we have $ \|f+g\|^q\leq \|f\|^q+\|g\|^q$.
\end{lemma}

\begin{proof}
Our proof is based on arguments from \cite{GurGur1971}.  Using homogeneity, and swapping the roles of $f$ and $g$ when necessary, we can restrict ourselves to vectors $f$ and $g$ with
\begin{enumerate}[label=($\dagger$)]
\item\label{cond:3} $\Vert g\Vert \le 1 =\Vert f\Vert$.
\end{enumerate}
Let $\delta_\XX$ denote modulus of convexity of the norm. Set
\[
\lambda:= 2(1-\delta_\XX( \varepsilon)),
\]
so that from the definition of $\delta_\XX$ we have $\|f+g\|\leq \lambda$ whenever $f$ and $g$ satisfy \ref{cond:2} and \ref{cond:3}. Note that $0<\lambda<2$. In the case when $\lambda\leq 1$, then for any $q> 1$ we have
\[
\|f+g\|^q\leq\lambda^q\leq 1 =\Vert f \Vert^q \leq \Vert f\Vert^q+\|g\|^q.
\]
In the case when $\lambda>1$, we pick $1 <q <\log_\lambda 2$, so that $1<\lambda^q<2$. Then, we choose
\[
\eta=1-(\lambda^q-1)^{1/q},
\]
so that $1+t^q\ge \lambda^q$ for every $t\ge 1-\eta$. Let $f$ and $g$ be vectors in $\XX$ satisfying \ref{cond:1}, \ref{cond:2} and \ref{cond:3}. We have
\[
\|f+g\|^q\leq \lambda^q\le 1+\Vert g\Vert^q =\Vert f\Vert^q+\Vert g\Vert^q.\qedhere
\]
\end{proof}

Building on Lemma~\ref{Lemma1}, we will achieve conditions under which series in $\ell_q(\XX)$, with $\XX$ superreflexive and $q>1$, converge. Prior to state and prove this result, we give an auxiliary lemma.

\begin{lemma}\label{lem:Woj}
Let $(f_n)_{n=1}^m$ be a finite family in a Banach space. Then, there is $0\le k \le m$ such that
\[
A_k:=\left| \Big\Vert \sum_{n=1}^k f_n\Big\Vert-\Big\Vert \sum_{n=k+1}^m f_n\Big\Vert\right|\le B:=\max_{1\le n\le m} \Vert f_n\Vert.
\]
\end{lemma}

\begin{proof}
If $\sum_{n=1}^m f_n=0$, then $A_k=0$ for all $0\le k\le m$. Suppose that $\sum_{n=1}^m f_n\not=0$. Then,
there is $0\le k\le m-1$ maximal with the property that $\Vert \sum_{n=1}^k f_n\Vert-\Vert \sum_{n=k+1}^m f_n\Vert\le 0$. We have 
\begin{align*}
2\min\{A_k,A_{k+1}\}
&\le A_k+A_{k+1}\\
&= - \left\Vert \sum_{n=1}^k f_n\right\Vert+\left\Vert \sum_{n=k+1}^m f_n\right\Vert
+\left\Vert \sum_{n=1}^{k+1} f_n\right\Vert-\left\Vert \sum_{n=k+2}^m f_n\right\Vert\\
&\le \Vert f_k\Vert + \Vert f_{k+1}\Vert \\
&\le 2B.
\end{align*}
Consequently, either $A_k\le B$ or $A_{k+1}\le B$.
\end{proof}

\begin{theorem}\label{lem:BW}
Let $\XX$ be a superreflexive Banach space. For every $C>0$ there are $K>0$ and $q>1$ such that if a sequence $(f_n)_{n=1}^\infty$ in $\XX$ satisfies $\sum_{n=1}^\infty \Vert f_n\Vert^q<\infty$ and 
\begin{equation}\label{eq:NonHomSchauder}
\left\Vert \sum_{n=j}^k f_n\right\Vert \le C \left\Vert\sum_{n=j}^k f_n - \sum_{n=k+1}^m f_n\right\Vert, \quad 1\le j \le k \le m,
\end{equation}
then $\sum_{n=1}^\infty f_n$ converges in $\XX$ and 
\[
\left\Vert \sum_{n=1}^\infty f_n\right\Vert\le K\left( \sum_{n=1}^\infty \Vert f_n\Vert^q\right)^{1/q}.
\]
\end{theorem}

\begin{proof}
Assume without lost of generality that $\XX$ is equipped with a uniformly convex norm. Fix $C>0$, and let $q>1$ and $0<\eta<1$ be as in Lemma~\ref{Lemma1} with respect to $ \varepsilon=1/(1+C)$. Setting $K:=2/\eta$, we  will prove that
\begin{equation}\label{eq:BW}
\left\Vert \sum_{n=j}^m f_n \right\Vert \le K \left( \sum_{n=j}^m \Vert f_n\Vert^q\right)^{1/q}
\end{equation}
for all integers $j,m$ with $1\le j\le m$. To that end, we proceed by induction on $m-j$. Note that \eqref{eq:BW} trivially holds if $m-j\in\{0,1\}$ because $K>2$. Suppose that the result holds whenever $m-j\in\{0,1,\ldots,r-1\}$, for some $r\ge2$, and take $j,m$ with $m-j=r$. Set
\[B:=\max_{j\le n \le m} \Vert f_n\Vert. 
\]
If $\Vert\sum_{n=j}^m f_n\Vert\le KB$, then \eqref{eq:BW} obviously holds; so let us assume that 
\[
\left\Vert\sum_{n=j}^m f_n\right\Vert>KB. 
\]
By Lemma~\ref{lem:Woj}, there is $j-1\le k\le m$ such that $| \Vert G_k\Vert-\Vert H_k\Vert |\le B$, where 
\[
G_k:=\sum_{n=j}^k f_n, \quad H_k:=\sum_{n=k+1}^m f_n. 
\]
Since $K\ge1$, note that in fact $j\le k\le m-1$. Hence both $k-j$ and $m-(k+1)$ are in $\{0,1,\ldots,r-1\}$. Setting $D:=\max\{\Vert G_k\Vert,\Vert H_k\Vert\}$ and $d:=\min\{\Vert G_k\Vert,\Vert H_k\Vert\}$, we have $D-d=| \Vert G_k\Vert-\Vert H_k\Vert |\le B$ and $B/\eta=KB/2<\Vert G_k+ H_k\Vert/2\le D$. Hence $d>(1-\eta)D$. By \eqref{eq:NonHomSchauder}, $\Vert G_k\Vert\le C\Vert G_k-H_k\Vert$, and it follows that $D\le(1+C)\Vert G_k-H_k\Vert$. Therefore, applying Lemma~\ref{Lemma1} and the induction hypothesis, we have 
\begin{align*}
\left\Vert \sum_{n=j}^mf_n\right\Vert^q&\le\Vert G_k\Vert^q+\Vert H_k\Vert^q\le K^q \sum_{n=j}^k \Vert f_n\Vert^q+ K^q\sum_{n=k+1}^m \Vert f_n\Vert^q\\&=K^q \sum_{n=j}^m \Vert f_n\Vert^q.
\end{align*}
Thus \eqref{eq:BW} is proved. We infer that $\sum_{n=1}^{\infty}f_n$ is a Cauchy series and the statement follows. 
\end{proof}

\begin{remark}
Theorem~\ref{lem:BW} gives that any semi-normalized Schauder basis $\XB$ of a superreflexive Banach space $\XX$ is $q$-Hilbertian for some $q>1$. Taking into account that $\XB^{**}$ is equivalent to $\XB$ (see \cite{AlbiacKalton2016}*{Corollary 3.2.4}), applying this result to $\XB^*$ and dualizing gives that $\XB$ also is $r$-Besselian for some $r<\infty$. Thus, Theorem~\ref{lem:BW} leads to a new proof of Theorem~\ref{thm:EJ}.
\end{remark}

\begin{remark}
The alternation of signs in the assumptions of Theorem~\ref{lem:BW} is essential. In fact, as we next show, the result does not hold under the assumption
\begin{equation}\label{eq:DW}
\left\Vert \sum_{n=j}^{k} f_n \right\Vert \le C \left\Vert \sum_{n=j}^m f_j \right\Vert, \quad 1\le j \le k \le m.
\end{equation}
Define $\xx_n=\ee_0+\ee_n$ for $n\in\NN$, where $(\ee_n)_{n=0}^\infty$ is the canonical basis in $\ell_2$. Biorthogonal functionals are $\xx_n^*=\ee_n^*$ for $n\in\NN$. Note that
\[
\frac1m\sum_{n=1}^m \xx_m= \ee_0+\frac{1}{m}\sum_{n=1}^m \ee_n
\]
converges in $\ell_2$-norm to $\ee_0$. This implies that $(\xx_n)_{n=1}^\infty$ is a basis of $\ell_2$ whose dual basis is norm-bounded. We have
\[
\left\|\sum_{n=j}^m \xx_n\right\|=\sqrt{(1+m-j)^2+1+m-j}, \quad 1\le j\le m.
\]
Consequently, for every $s\in\NN$, the sequence $(f_{s,n})_{n=1}^\infty$ defined by $f_{s,n}=\xx_n$ if $n\le s$ and $f_{s,n}=0$ otherwise satisfies \eqref{eq:DW} with $C=1$. However, for any $q>1$,
\[
\frac{\left\|\sum_{n=1}^\infty f_{s,n}\right\|}{ \left(\sum_{n=1}^\infty \Vert f_{s,n}\Vert^q\right)^{1/q}}
=\frac{\left\|\sum_{n=1}^s \xx_n\right\|}{ \left(\sum_{n=1}^s \Vert \xx_n\Vert^q\right)^{1/q}}
=\frac{\sqrt{s^2+s}}{\sqrt{2} \, s^{1/q}}
\xrightarrow[s \to \infty]{}\infty.
\]
\end{remark}

\section{Superreflexive Lorentz sequence spaces}\label{sect:SRLSS}\noindent
To undertake our study of Lorentz sequence spaces, we need to introduce some terminology. A sequence $(t_n)_{n=1}^\infty$ of positive scalars is said to be \emph{essentially increasing} if
\[
\sup_{n\le m} \frac{t_n}{t_m}<\infty.
\]
The sequence $(t_n)_{n=1}^\infty$ is essentially increasing if and only if it is equivalent to a non-decreasing sequence of positive scalars.

Following \cite{DKKT2003}, we say that $(t_n)_{n=1}^\infty$ has the \emph{upper regularity property} (URP for short) if there is an integer $b\ge 2$, such that
\[
t_{bn}\le \frac{b}{2} t_n, \quad n\in\NN.
\]
The following lemma is essentially known.

\begin{lemma}[cf.\ \cite{AlbiacAnsorena2016}*{Lemma 2.12}]\label{lem:URP}
Consider the following conditions associated with a sequence $\primt=(t_n)_{n=1}^\infty$ of positive scalars.
\begin{enumerate}[label=(\alph*), leftmargin=*,widest=c]
\item\label{URP:1} $\primt$ has the URP.
\item\label{URP:2} There is $0<a<1$ such that $(n^a/t_n)_{n=1}^\infty$ is essentially increasing.
\item\label{URP:3} There is a constant $C$ such that
\[
u_m:=\frac{1}{m}\sum_{n=1}^m \frac{1}{t_n}\le C \frac{1}{t_m}, \quad m\in\NN.
\]
\end{enumerate}
Then \ref{URP:1} implies \ref{URP:2}, and \ref{URP:2} implies \ref{URP:3}. Moreover, if $\primt$ is essentially increasing, then \ref{URP:3} implies \ref{URP:1}, and $u_m\approx 1/t_m$ for $m\in\NN$.
\end{lemma}

\begin{lemma}\label{lem:ImprovedURP}
Suppose that the sequence $(t_n)_{n=1}^\infty$ is essentially increasing and has the URP. Then, there is $p_0>1$ such that $(t_n^p)_{n=1}^\infty$ has the URP for every $1< p < p_0$.
\end{lemma}

\begin{proof}
By Lemma~\ref{lem:URP}, there is $0<a<1$ such that $(n^a/t_n)_{n=1}^\infty$ is essentially increasing. Set $p_0:=1/a$. If $1<p<p_0$, then $0<ap<1$ and $(n^{ap}/t_n^p)_{n=1}^\infty$ is essentially increasing. Applying again Lemma~\ref{lem:URP} gives that $(t_n^p)_{n=1}^\infty$ has the URP.
\end{proof}

We also need the dual property of the URP. We say that $(t_n)_{n=1}^\infty$ has the \emph{lower regularity property} (LRP for short) if there is an integer $b\ge 2$ such
\begin{equation*}
2 t_m \le t_{bm}, \quad m\in\NN.
\end{equation*}
Given a sequence $\primt=(t_n)_{n=1}^\infty$ of positive scalars, we define its \emph{dual sequence} as
$
\primt^*=(n/t_n)_{n=1}^\infty
$.

\begin{lemma}\label{lem:DualURPInc}
Suppose that the sequence $\primt=(t_n)_{n=1}^\infty$ has the URP. Then $\primt^*$ is essentially increasing.
\end{lemma}

\begin{proof}
Use Lemma~\ref{lem:URP} to pick $0<a<1$ such that $(n^a/t_n)_{n=1}^\infty$ is essentially increasing. Then $\primt^*=(n^{1-a}\, n^a/t_n)_{n=1}^\infty$ is essentially increasing.
\end{proof}

The following lemma, which we record for further reference, is clear from definitions.
\begin{lemma}\label{lem:DualRegular}
A sequence has the URP if and only if its dual sequence has the LRP.
\end{lemma}

It will be convenient in some situations to use a different notation for Lorentz sequence spaces. Given $0<q<\infty$ and a weight $\uu=(u_n)_{n=1}^\infty$, we set
\[
\Vert f\Vert_{(\uu,q)}:=\left( \sum_{n=1}^\infty a_n^q u_n\right)^{1/q}, f\in c_0, \, f^*=(a_n)_{n=1}^\infty,
\]
and we denote by $d(\uu,q)$ the space consisting of all $f\in c_0$ such that $\Vert f\Vert_{(q,\uu)}<\infty$. The following lemma relates the spaces $d_{1,q}(\ww)$ and $d(\uu,q)$.

\begin{lemma}[see \cite{AABW2021}*{\S9.2}]\label{lem:ChangeScale}
Let $0<q<\infty$, and let $\ww=(w_n)_{n=1}^\infty$ and $\uu=(u_n)_{n=1}^\infty$ be weights with primitive sequences $(s_n)_{n=1}^\infty$ and $(t_n)_{n=1}^\infty$ respectively. Then $\Vert \cdot \Vert_{(\uu,q)}\approx \Vert \cdot \Vert_{1,q,\ww}$ if and only if $s_n^q\approx t_n$ for $n\in\NN$.
\end{lemma}

Although it is customary to designate Lorentz sequence spaces after the weight $\ww$, it must be conceded that, as Lemma~\ref{lem:ChangeScale} exhibits, they depend on its primitive sequence $(s_m)_{m=1}^\infty$ rather than on $\ww$.

\begin{lemma}\label{lem:LRPEquivalence}
Let $0<q<\infty$, and let $\ww=(w_n)_{n=1}^\infty$ be a weight whose primitive sequence $\prim=(s_n)_{n=1}^\infty$ has the LRP. Suppose also that the dual sequence $\prim^*$ is essentially increasing.
\begin{enumerate}[label=(\roman*), leftmargin=*,widest=ii]
\item\label{LRP:a} There is a non-increasing weight $\vv=(v_n)_{n=1}^\infty$ such that
\[
m v_m \approx \sum_{n=1}^m v_n \approx s_m, \quad m\in\NN.
\]
\item\label{LRP:b} If $\uu=(s_n^q/n)_{n=1}^\infty$, then $ \Vert \cdot \Vert_{1,q,\ww}\approx \Vert \cdot \Vert_{(\uu,q)}$.
\end{enumerate}
\end{lemma}

\begin{proof}
Pick a non-increasing weight $\vv$ equivalent to $(s_n/n)_{n=1}^\infty$. By Lemma~\ref{lem:DualRegular}, $\prim^*$ has the URP. Applying Lemma~\ref{lem:URP} to $\prim^*$ we obtain
\[
m v_m \approx s_m\approx \sum_{n=1}^m \frac{s_n}{n} \approx t_m:=\sum_{n=1}^m v_n , \quad m\in\NN.
\]
Consequently, $\Vert \cdot \Vert_{1,q,\ww}\approx \Vert \cdot \Vert_{1,q,\vv}$. In turn, since
\[
t_n^{q-1}v_n\approx s_n^{q-1} \frac{s_n}{n}=\frac{s_n^q}{n}, \quad n\in\NN,
\]
we have $\Vert \cdot \Vert_{1,q,\vv}\approx \Vert \cdot \Vert_{(\vv,q)}$.
\end{proof}

The \emph{discrete Hardy operator} $\HD\colon\FF^\NN\to \FF^\NN$ is defined by
\[
f=(a_n)_{n=1}^\infty \mapsto \HD(f):= \left( \frac{1}{m} \sum_{n=1}^m a_n \right)_{m=1}^\infty.
\]

\begin{lemma}\label{lem:BHO}
Let $1<q<\infty$ and let $\ww$ be a weight whose primitive sequence has the URP. Then
\[
\Vert f \Vert_{1,q,\ww} \approx \Vert \HD(f^*) \Vert_{1,q,\ww}, \quad f\in \FF^\NN.
\]
\end{lemma}

\begin{proof}
Just combine \cite{CRS2007}*{Corollary 1.3.9} with Lemmas~\ref{lem:URP} and \ref{lem:ChangeScale}.
\end{proof}

Given $0<q\le\infty$ and a weight $\ww=(w_n)_{n=1}^\infty$, $\Vert \cdot \Vert_{1,q,\ww}$ is a quasi-norm if and only if the primitive sequence $(s_n)_{n=1}^\infty$ of $\ww$ is doubling, in which case $d_{1,q}(\ww)$ is a quasi-Banach space (see \cite{CRS2007}*{Theorem 2.2.13}). We have $d_{1,q}(\ww)=c_0$ if and only if $\prim$ is bounded. Apart from this trivial case, the unit vector system $\UVS=(\ee_n)_{n=1}^\infty$ generates the whole space $d_{1,q}(\ww)$
if and only if $q<\infty$, in which case $\UVS$ is a boundedly complete basis. It is known when $d_{1,q}(\ww)$ is locally convex, i.e., a Banach space.

\begin{theorem}[see \cite{CRS2007}*{Theorems~2.5.10 and 2.5.11}]\label{thm:LC}
Let $\ww$ be a weight whose primitive sequence $\prim$ is doubling.
\begin{enumerate}[label=(\roman*),leftmargin=*,widest=iii]
\item Given $0<q<1$, $d_{1,q}(\ww)$ is locally convex if and only if $\prim$ is bounded.
\item $d_{1,1}(\ww)$ is locally convex if and only if $\prim^*$ is essentially increasing.
\item Given $1<q\le \infty$, $d_{1,q}(\ww)$ is locally convex if and only $\prim$ has the URP.
\end{enumerate}
\end{theorem}

In the case when $0<q\le 1$ or $q=\infty$, the dual spaces of the separable part of Lorentz sequence spaces has been clearly identified. We denote by $\ls_{1,\infty}(\ww)$ the subspace of $d_{1,\infty}(\ww)$ generated by the unit vectors. Let $m(\ww)$ be the Marcinkiewicz space consisting of all $f\in c_0$ whose non-increasing rearrangement $f^*=(a_n)_{n=1}^\infty$ satisfies
\[
\Vert f\Vert_{m(\ww)}=\sup_m \frac{1}{s_m}\sum_{n=1}^m a_n<\infty,
\]
where $\prim=(s_n)_{n=1}^\infty$ is the primitive sequence of $\ww$. By Lemmas~\ref{lem:URP} and
\ref{lem:DualRegular}, if $\prim$ has the LRP and $\ww^*$ is a weight whose primitive sequence is equivalent to $\prim^*$, then $m(\ww)=d_{1,\infty}(\ww^*)$. If $\prim$ is bounded, then $m(\ww)=\ell_1$. In turn, if $\prim^*$ is bounded, then $m(\ww)=c_0$.

\begin{theorem}[see \cite{CRS2007}*{Theorems 2.4.14 and 2.5.10}]\label{thm:dualA}
Let $\ww$ be a weight whose primitive sequence $\prim$ is doubling and unbounded.
\begin{enumerate}[label=(\roman*), leftmargin=*, widest=ii]
\item\label{dualA:a} If $0<q\le 1$, then $(d_{1,q}(\ww))^*=m(\ww)$; and
\item $(\ls_{1,\infty}(\ww))^*=d_{1,1}(1/\prim)$.
\end{enumerate}
\end{theorem}

In the case when $1<q<\infty$, such a simple characterization of the dual space of $d_{1,q}(\ww)$ is not possible in general (see \cite{Garling1969}*{Theorem 1} and \cite{CRS2007}*{Theorems 2.4.14 and 2.5.10}). Notwithstanding, it can be identified to be a Lorentz sequence space under certain regularity conditions.

\begin{theorem}[cf.\ \cite{Allen1978}*{Theorem 1.1}]\label{thm:dualB}
Let $1<q<\infty$, and let $\ww$ be a weight whose primitive sequence $\prim=(s_n)_{n=1}^\infty$ has both the URP and the LRP. Then $(d_{1,q}(\ww))^*=d_{1,q'}(\ww^*)$, where $\ww^*$ is a weight whose primitive sequence is equivalent to $\prim^*$.
\end{theorem}

\begin{proof}
Set $t_n=n/s_n$ for all $n\in\NN$. By Lemma~\ref{lem:LRPEquivalence}\ref{LRP:a} we can assume without lost of generality that $w_n \approx 1/t_n$ for $n\in\NN$. By Lemma~\ref{lem:ChangeScale} and \cite{CRS2007}*{Theorems 2.4.14 and 2.5.10}, the dual space of $d_{1,q}(\ww)$ is, under the natural pairing, the space
\[
\YY=\{ f\in\FF^\NN \colon \HD(f^*)\in d(\uu,q')\},
\]
where the weight $\uu=(u_n)_{n=1}^\infty$ is given by $u_1=s_1^{-q'} $ and
\[
u_n=n^{q'}\left(s_{n-1}^{-q'}-s_n^{-q'}\right), \quad n\ge 2.
\]
The equivalence $x^{-q'}-1\approx 1-x$ for $0< x \le 1$ yields
\[
u_n \approx n^{q'} s_n^{-q'-1} w_n \approx t_n^q \frac{1}{n} , \quad n\in\NN.
\]
Hence, applying Lemma~\ref{lem:LRPEquivalence}\ref{LRP:b} to $\ww^*$ and $q'$ gives
\[
\left( \sum_{n=1}^m u_m\right)^{1/q'} \approx \ff[d_{1,q'}(\ww^*)] \approx t_m, \quad m\in\NN.
\]
Therefore, by Lemma~\ref{lem:ChangeScale}, $ d(\uu,q')=d_{1,q'}(\ww^*)$. Applying Lemma~\ref{lem:BHO} puts an end to the proof.
\end{proof}

We emphasize that Allen \cite{Allen1978} dealt with Lorentz sequence spaces $d(\uu,q)$ with $\uu$ non-increasing. Since there are weights $\ww$ such that $d_{1,q}(\ww)$ can not be expressed as $d(\uu,q)$ with $\uu$ non-increasing, Theorem~\ref{thm:dualB} is more general that \cite{Allen1978}*{Theorem 1.1}. Similarly, we will study superreflexivity for a broader class of spaces than that dealt with by Altshuler \cite{Altshuler1975}. To address this task, we will take advantage of the lattice structure these spaces are equipped with, which allows us to consider $p$-convexifications of Lorentz sequence spaces. Recall that, given a quasi-Banach lattice $\YY$ and $0<p<\infty$, the $p$-convexified space $\YY^{(p)}$ is the quasi-Banach lattice defined by the quasi-norm
$f\mapsto \Vert |f|^p \Vert^{1/p}$.

\begin{lemma}\label{lem:SLSConvexified}
Let $\ww$ be a weight whose primitive sequence $\prim$ is doubling, and let $p$, $q\in(0,\infty)$. Then,
\[
\left( d_{1,q}(\ww)\right)^{(p)}=d_{1,pq}(\ww_{1/p})
\]
up to an equivalent quasi-norm, where $\ww_{1/p}$ is the weight whose primitive sequence is $\prim^{1/p}$.
\end{lemma}

\begin{proof}
It is clear that $(d(\uu,q))^{(p)}=d(\uu,pq)$ for every weight $\uu$. Then, the result follows from Lemma~\ref{lem:ChangeScale}.
\end{proof}

\begin{proposition}\label{prop:pconvex}
Let $1<q<\infty$, and let $\ww$ be a weight whose primitive sequence $\prim$ has the URP. Then $d_{1,q}(\ww)$ is a $p$-convex lattice for some $p>1$.
\end{proposition}

\begin{proof}
By Lemma~\ref{lem:ImprovedURP}, there is $1<p<q$ be such that $\prim^p$ has the URP. Let $\ww_p$ be the weight whose primitive sequence is $\prim^p$. By Theorem~\ref{thm:LC}, $d_{1,q/p}(\ww_p)$ is locally convex, i.e., a $1$-convex quasi-Banach lattice. Hence, $(d_{1,q/p}(\ww_p))^{(p)}$ is a $p$-convex lattice. Applying Lemma~\ref{lem:SLSConvexified} puts an end to the proof.
\end{proof}

\begin{proposition}\label{prop:Marloo}
Let $\ww$ be a weight whose primitive sequence $\prim=(s_n)_{n=1}^\infty$ is unbounded. Then the unit vector system of $m(\ww)$ has a block basic sequence equivalent to the unit vector system of $\ell_\infty$.
\end{proposition}
\begin{proof}
Assume without lost of generality that $\prim^*$ is unbounded. Pick an arbitrary $\lambda>1$. We recursively construct a sequence $(n_k)_{k=1}^\infty$ such that
\[
\sum_{j=1}^{k-1} s_{n_j} \le (\lambda-1) s_{n_{k}}, \quad
\]
and $(n_k/s_{n_k})_{k=1}^\infty$ is non-decreasing. Set, with the convention $m_0=0$,
\[
\xx_k= \frac{s_{m_k}}{m_k} \sum_{n=1+m_{k-1}}^{m_k} \ee_n, \quad k\in\NN.
\]
A routine check gives that the block basic sequence $(\xx_k)_{k=1}^\infty$ is $\lambda$-equivalent to the unit vector system of $\ell_\infty$.
\end{proof}

To properly understand the following result, we must concede that we can define reflexivity (or superreflexivity) on quasi-Banach spaces, but reflexive spaces end up being locally convex.

\begin{theorem}\label{thm:LorentzReflexive}
Let $\ww$ be a weight whose primitive sequence $\prim$ is doubling, and let $0<q\le \infty$. Then $d_{1,q}(\ww)$ is reflexive if and only if $1<q<\infty$ and $\prim$ has the URP.
\end{theorem}

\begin{proof}
Combining Theorem~\ref{thm:LC}, Theorem~\ref{thm:dualA} and Proposition~\ref{prop:Marloo} yields the only if part. Assume that $1<q<\infty$ and $\prim$ has the URP. By Proposition~\ref{prop:pconvex}, there is $p>1$ such that the unit vector system of $\ell_p$ dominates every block basic sequence of the unit vector system of $\YY:=d_{1,q}(\ww)$. Consequently, no block basic sequence of the unit vector system of $\YY$ is equivalent to the unit vector system of $\ell_1$. By a classical theorem of James (see \cite{AlbiacKalton2016}*{Theorem 3.3.1}), the unit vector system is a shrinking basis of $\YY$. Since it is also boundedly complete, applying \cite{AlbiacKalton2016}*{Theorem 3.2.19} gives that $\YY$ is reflexive.
\end{proof}

To study superreflexivity in Lorentz sequence spaces, we will use the following important result.
\begin{theorem}[see \cite{LinTza1979}*{\S1.f}]\label{thm:SRLattice}
Let $\YY$ be a Banach lattice. Then, the following are equivalent.
\begin{enumerate}[label=(\roman*),leftmargin=*,widest=iii]
\item There are $1<q<r<\infty$ such that $\YY$ is a $p$-convex lattice and a $r$-concave lattice.
\item $\YY$ is a superreflexive Banach space.
\item There is $q>1$ such that $\YY$ has Rademacher type $q$.
\end{enumerate}
\end{theorem}

\begin{theorem}[cf.\ \cite{Altshuler1975}]
Let $\ww$ be a weight whose primitive sequence $\prim$ is doubling, and let $0<q\le \infty$, let $0<q\le \infty$. Then, $d_{1,q}(\ww)$ is superreflexive if and only if $1<q<\infty$ and $\prim$ has both the LRP and the URP.
\end{theorem}
\begin{proof}
If $q\notin(1,\infty)$, then $d_{1,q}(\ww)$ is not reflexive by Theorem~\ref{thm:LorentzReflexive}. If $\YY=d_{1,q}(\ww)$ is superreflexive, applying Theorem~\ref{thm:SRLattice} yields $p>1$ such that $\YY$ has type $p$. Taking into account that the unit vector system is an almost greedy basis of $\YY$, we infer from \cite{DKKT2003}*{Proposition 4.1}, that $\ff[\YY]$ has both the LRP and the URP. Since $\prim$ is equivalent to $\ff[\YY]$, it also has the the LRP and the URP.

Assume that $1<q<\infty$ and that $\prim=(s_n)_{n=1}^\infty$ has the LRP and the URP. By Lemma~\ref{lem:DualRegular}, the dual sequence $\prim^*$ has both the URP and the LRP. By Proposition~\ref{prop:pconvex}, there is $p>1$ such that both $\YY$ and $d_{1,q'}(\ww^*)$ are $p$-convex lattices, where $\ww^*$ is a weight whose primitive sequence is equivalent to $\prim^*$. Combining Theorem~\ref{thm:dualB} with \cite{LinTza1979}*{Theorem 1.f.7} gives that $\YY$ is an $r$-concave lattice for all $r>p'$. Applying Theorem~\ref{thm:SRLattice} puts an end to the proof.
\end{proof}

\section{Embeddings associated with quasi-greedy bases}\label{sect:EQB}\noindent
A basis $\XB$ of a quasi-Banach space $\XX$ is said to be \emph{quasi-greedy} if the TGA with respect to it is uniformly bounded. If $\XB$ is quasi-greedy, its quasi-greedy constant will be the smallest constant $C$ such that
\[
\max|\{ \Vert f-\GG_m(f)\Vert, \Vert \GG_m(f) \Vert\} \le C \Vert f \Vert, \quad f\in\XX, \; m\in\NN.
\]
It is known that any basis $\XB$ of any $p$-Banach space $\XX$, $0<p\le 1$, satisfies the estimate
\begin{equation*}
d_{1,p}(\ww_u) \stackrel{\XB}\hookrightarrow \XX.
\end{equation*}
where $\ww_u$ is the weight whose primitive sequence is $\udf[\XB,\XX]$ (see \cite{AABW2021}*{Theorem 9.12}. In the case when $\XX$ is a superreflexive Banach space and $\XB$ is quasi-greedy this embedding can be improved.

\begin{theorem}\label{prop:embedding1}
Suppose $\XB$ is a quasi-greedy basis of a superreflexive Banach space $\XX$. Let $\ww_u$ be the weight whose primitive sequence is $\udf[\XB,\XX]$. Then, there is $q>1$ such that
\[
d_{1,q}(\ww_u) \stackrel{\XB}\hookrightarrow \XX.
\]
\end{theorem}

\begin{proof}
If $m\in\NN$, $A\subseteq\NN$ has cardinality at most $m$, and $(b_{n})_{n\in A}$ are scalars with $|b_n|\le 1$ for all $n\in A$, then
\begin{equation}\label{eq:embedding1}
\left\Vert \sum_{n\in A} b_n\, \xx_n\right\Vert \le s_m:=\udf[\XB,\XX](m)
\end{equation}
(see \cite{AABW2021}*{Corollary 2.3}). Let $\ww_q=(u_{n})_{n=1}^\infty$ be the weight whose primitive sequence is $(s_n^q)_{n=1}^\infty$. Use Lemma~\ref{lem:ChangeScale} to pick a constant $D$ such that
\[
\Vert f\Vert_{(\ww_q,q)}\le D \Vert f\Vert_{1,q,\ww_u}, \quad f\in c_0.
\]
Let $q\in(1,\infty)$ and $K\in(0,\infty)$ be as in Lemma~\ref{lem:BW} with respect to $\XX$ and the quasi-greedy constant $C$ of $\XB$.

Let $(a_n)_{n=1}^\infty\in c_{00}$ be such that $(|a_n|)_{n=1}^\infty$ is non-increasing. Put $t=|a_1|$. Consider for each $k\in\NN$ the set
\[
J_k=\{n\in\NN \colon t 2^{-k}< |a_n| \le t 2^{-k+1}\}
\]
and the vector
\[
f_k=\sum_{n\in J_k} a_n\, \xx_n.
\]
Notice that $(J_k)_{k=1}^\infty$ is a partition of $\{n\in\NN \colon a_n\not=0\}$ and that $(f_k)_{k=1}^\infty$ satisfies \eqref{eq:NonHomSchauder}. Set $n_k=|J_k|$ ($n_0=0$) and $m_k=\sum_{j=1}^k n_j$, so that $J_k=\{n\in\NN \colon m_{k-1}+1\le n \le m_k\}$.
Combining \eqref{eq:embedding1} with Abel's summation formula gives
\begin{align*}
\left\Vert \sum_{n=1}^\infty a_n \, \xx_n\right\Vert^q
&=K^q\left\Vert\sum_{k=1}^\infty f_k \right\Vert^q\\
&\le K^q \sum_{k=1}^\infty \Vert f_k\Vert^q\\
&\le K^q \sum_{k=1}^\infty (t 2^{-k+1} s_{m_k})^q \\
&= K^q (2t)^q \sum_{k=1}^\infty 2^{-kq} \sum_{j=1}^k \sum_{n\in J_j} u_n\\
&= \frac{ (2tK)^q}{1-2^{-q}} \sum_{j=1}^\infty 2^{-jq} \sum_{n\in J_j} u_n\\
&\le \frac{(4K)^q}{2^q -1} \sum_{n=1}^\infty |a_n|^q u_n.
\end{align*}
Consequently, for every $f=(a_n)_{n=1}^\infty\in c_{00}$ we have
\[
\left\Vert \sum_{n=1}^\infty a_n \, \xx_n\right\Vert
\le 4 K (2^q-1)^{-1/q} \Vert f \Vert_{(\ww_q,q)}
\le 4 K D(2^q-1)^{-1/q}\Vert f \Vert_{1,q,\ww_u}.
\]
Since the unit vectors generate the whole space $d_{1,q}(\ww_u)$, we are done.
\end{proof}

We are almost ready to tackle the proof of Theorem~\ref{thm:main}. Before we do that, we record a result that we will need. A basis $\XB$ of a quasi-Banach space $\XX$ is said to be \emph{bidemocratic} if
\[
\sup_m \frac{1}{m} \udf[\XB,\XX](m) \udf[\XB^*,\XX^*](m)<\infty.
\]

\begin{proposition}\label{prop:SSBD}
Let $\XB$ be a quasi-greedy basis (or, more generally, a basis with the property that the restricted truncation operators are uniformly bounded) of a Banach space $\XX$. Suppose that $\XB$ is democratic and that $\XX$ has type $q>1$. Then $\udf[\XB,\XX]$ has the URP and the LRP, and $\XB$ is bidemocratic.
\end{proposition}

\begin{proof}
By the proof of \cite{DKKT2003}*{Proposition 4.1} (which works not only for almost greedy bases but also for super-democratic ones), $\udf[\XB,\XX]$ has the URP and the LRP. Then, combining \cite{AABW2021}*{Proposition 10.17} with Lemma~\ref{lem:URP} (or, in the case that $\XB$ is quasi-greedy, using \cite{DKKT2003}*{Proposition 4.4}), gives that $\XB$ is bidemocratic.
\end{proof}

\begin{proof}[Proof of Theorem~\ref{thm:main}]
By Proposition~\ref{prop:SSBD} and Lemma~\ref{lem:DualRegular}, $\prim:=\udf$ and $\prim^*$ have the URP and the LRP. Moreover, $\XB$ is bidemocratic. Let $\ww^*$ be a weight whose primitive sequence is equivalent to $\prim^*$. Combining Proposition~\ref{prop:SSBD} with \cite{DKKT2003}*{Theorem 5.4}, gives that dual basis $\XB^*$ is also almost greedy, with fundamental function of the same order as $\prim^*$. Consequently, since $\XX^*$ is also superreflexive, applying Theorem~\ref{prop:embedding1} yields $s$ and $q>1$ such that
\[
d_{1,q}(\ww) \stackrel{\XB}\hookrightarrow \XX
\; \text{ and }\;
d_{1,s}(\ww^*) \stackrel{\XB^*}\hookrightarrow \XX^*.
\]
Taking into account Theorem~\ref{thm:dualB}, dualizing the second embedding gives
\[
\XX \stackrel{\XB^{**}}\hookrightarrow d_{1,s'}(\ww).
\]
Since $\XB^{**}$ is equivalent to $\XB$ by \cite{AABW2021}*{Theorem 10.15}, we are done.
\end{proof}

We conclude with the ready consequence of applying Theorem~\ref{thm:main} to a quasi-greedy basis of a Hilbert space. Note that this result improves \cite{Woj2000}*{Theorem 3}.

\begin{theorem}\label{thm:SSImprovedHilbert}
Let $\XB$ be a quasi-greedy basis $\XB$ of $\ell_2$. Then, there are $1<q<r<\infty$ such that
\[
\ell_{2,q} \stackrel{\XB}\hookrightarrow \ell_2 \stackrel{\XB}\hookrightarrow \ell_{2,r}.
\]
\end{theorem}

\section{Open problems}\label{sect:OP}\noindent
Any quasi-greedy basis $\XB$ of any quasi-Banach space $\XX$ satisfies the embedding
\[
\XX \stackrel{\XB}\hookrightarrow d_{1,\infty}(\ww_l),
\]
where $\ww_l$ is the weight whose primitive sequence is $\ldf[\XB,\XX]$ (see \cite{AABW2021}*{Theorem 9.12}). In light of Theorems~\ref{thm:main} and \ref{prop:embedding1}, it is natural to wonder whether this embedding can also be improved in the particular case that $\XX$ is a superreflexive space.

\begin{question}\label{conj:Anso}
Let $\XB$ be a quasi-greedy basis of a superreflexive Banach space $\XX$. Let $\ww_l$ denote the weight whose primitive sequence is $\ldf[\XB,\XX]$. Is there $r<\infty$ such that $ \XX \stackrel{\XB}\hookrightarrow d_{1,r}(\ww_l)$?
\end{question}

We point out that the answer to Question~\ref{conj:Anso} is positive if we restrict ourselves to the more demanding class of unconditional bases. Indeed, if $\XB$ is a semi-normalized unconditional bases of $\XX$, then, by Theorem~\ref{thm:SRLattice}, the lattice structure it induces on $\XX$ is $r$-concave for some $r<\infty$. Therefore, $ \XX \stackrel{\XB}\hookrightarrow d_{1,r}(\ww_l)$ (see \cite{AlbiacAnsorena2021}*{Theorem 7.3}).

Thus, Question~\ref{conj:Anso} connects with the topic of finding conditional quasi-greedy basis in Banach spaces. To the best of our knowledge, all known methods for building conditional quasi-greedy bases give, when applied to superreflexive spaces, almost greedy bases (see \cites{KoTe1999,Woj2000,DKK2003,GHO2013,GW2014, BBGHO2018,AADK2019b,AABW2021}). Of course, we can combine conditional almost greedy bases through direct sums to reach examples of conditional quasi-greedy non-democratic bases, but this procedure does not lead no a negative answer to Question~\ref{conj:Anso}. As we have already noted, the dual basis of an almost greedy basis of a superreflexive space is quasi-greedy. In this regard, even the following very natural question seems to be open.

\begin{question}\label{conj:DualQG}
Let $\XB$ a quasi-greedy basis of $\ell_p$, $p\in(1,2)\cup(2,\infty)$. Is its dual basis $\XB^*$ quasi-greedy?
\end{question}
Notice that the answer to Question~\ref{conj:DualQG} is positive for $p=2$, but the proof goes through democracy.

Finally, we wonder whether the assumption that $\XB$ is almost greedy in Theorem~\ref{thm:main} can be relaxed. Bases $\XB$ that can be squeezed between two symmetric sequence spaces with equivalent fundamental functions are characterized as those democratic bases for which the restricted truncation operators are uniformly bounded (see \cite{AABW2021}*{Lemma 9.3 and Corollary 9.15}). Moreover, if the space $\XX$ is superreflexive, such bases satisfy a stronger condition. Namely, they are bidemocratic by Proposition~\ref{prop:SSBD}. If we focus on Hilbert spaces, the following question arises from Theorem~\ref{thm:SSImprovedHilbert}.

\begin{question}
Let $\XB$ be a bidemocratic basis of $\ell_2$. Are there $1<q<r<\infty$ such that
$
\ell_{2,q} \stackrel{\XB}\hookrightarrow \ell_2 \stackrel{\XB}\hookrightarrow \ell_{2,r}
$?
\end{question}

%\bibliography{Biblio}{}
%\bibliographystyle{plain}

% \bib, bibdiv, biblist are defined by the amsrefs package.

\end{document}